\documentclass[12pt]{amsart}
\usepackage{amssymb,enumerate}
\usepackage{amsmath}
\usepackage{bbm}           
\usepackage{bm}
\usepackage{eso-pic,graphicx}
\usepackage{tikz}

\usepackage[colorlinks=true, pdfstartview=FitV, linkcolor=blue, citecolor=blue, urlcolor=blue]{hyperref}

\makeatletter
\def\Ddots{\mathinner{\mkern1mu\raise\p@
\vbox{\kern7\p@\hbox{.}}\mkern2mu
\raise4\p@\hbox{.}\mkern2mu\raise7\p@\hbox{.}\mkern1mu}}
\makeatother

\def\Xint#1{\mathchoice
{\XXint\displaystyle\textstyle{#1}}%
{\XXint\textstyle\scriptstyle{#1}}%
{\XXint\scriptstyle\scriptscriptstyle{#1}}%
{\XXint\scriptscriptstyle\scriptscriptstyle{#1}}%
\!\int}
\def\XXint#1#2#3{{\setbox0=\hbox{$#1{#2#3}{\int}$}
\vcenter{\hbox{$#2#3$}}\kern-.5\wd0}}

\def\dashint{\Xint-}

\newtheorem{theorem}{Theorem}[section]

\theoremstyle{definition}

\newtheorem{lemma}[theorem]{Lemma}

\def\al{{\alpha}}
\def\R{\mathbb R}
\def\N{\mathbb N}

\def\ra{\rightarrow}
\def\bey{\begin{eqnarray*}}
\def\eey{\end{eqnarray*}}

\DeclareMathOperator{\supp}{supp}

\def\M{{\mathcal M}}

\newcommand{\cz}{Calder\'on-Zygmund}

\begin{document}


\title[Compact bilinear commutators: The weighted case]{Compact bilinear commutators: The weighted case}

\date{\today}

\author[\'A. B\'enyi]{\'Arp\'ad B\'enyi}

\address{%
Department of Mathematics \\
Western Washington University\\
516 High Street\\
Bellingham, WA 98225, USA}
\email{arpad.benyi@wwu.edu}

\author[W. Dami\'an]{Wendol\'in Dami\'an}
\address{%
Departamento de An\'alisis Matem\'atico, Facultad de Matem\'aticas \\
Universidad de Sevilla\\
41080 Sevilla, Spain}
\email{wdamian@us.es}

\author[K. Moen]{Kabe Moen}
\address{%
Department of Mathematics\\
University of Alabama \\
Tuscaloosa, AL 35487, USA}
\email{kabe.moen@ua.edu}

\author[R. H. Torres]{Rodolfo H. Torres}
\address{%
Department of Mathematics\\
University of Kansas\\
Lawrence, KS 66045, USA}
\email{torres@math.ku.edu}

\thanks{\'A. B. partially supported by a grant from the Simons Foundation (No.~246024). W. D. is supported by the Junta de Andaluc\'ia (P09-FQM-47459) and the Spanish Ministry of Science and Innovation (MTM2012-30748).  K. M. and R.H. T. partially supported by NSF grants DMS 1201504 and DMS 1069015, respectively.}

\begin{abstract}
{Commutators of bilinear \cz\, operators and multiplication by functions in a certain subspace of the space of functions of bounded mean oscillations are shown to be compact on appropriate products of weighted Lebesgue spaces.}
\end{abstract}

\maketitle

\section{Introduction and statements of main results}

In \cite{BT}, B\'enyi and Torres revisited a notion of compactness in a bilinear setting, which was first introduced by Calder\'on in his fundamental paper on interpolation \cite{C}. They showed in \cite{BT} that commutators of bilinear \cz\, operators with multiplication by $CMO$ functions are compact bilinear operators from $L^{p_1}\times L^{p_2}\to L^p$ for $1<p_1, p_2<\infty$ and $1/p_1+1/p_2=1/p \leq 1$, thus giving an extension of a classical result of Uchiyama \cite{U} from the linear setting. In a subsequent joint work with Dami\'an and Moen \cite{BDMT}, the scope of the notion of compactness was expanded to include the commutators of a larger family of operators that contains bilinear \cz\, ones, as well as several singular bilinear fractional integrals. All these compactness results rely on the Frech\'et-Kolmogorov-Riesz characterization of precompact sets in unweighted Lebesgues spaces $L^p$, see Yosida's book \cite[p. 275]{Y} and the expository note of Hanche-Olsen and Holden \cite{HOH}.

What happens if we change the Lebesgue measure $dx$ with weighted versions $wdx$? This article originates in this natural question. Although seemingly simple, the answer to this question turns out to be more delicate than in the unweighted case. As we shall see, the compactness on products of weighted Lebesgue spaces depends rather crucially on the class of weights $w$ considered. We note that in the linear case the compactness of the 
commutator on weighted spaces was not known until the recent work of Clop and Cruz \cite{CC}. We will rely on their work for the selection of weights and some computations. 

Let then $T$ be a bilinear Calder\'on-Zygmund operator.  For the purposes of this article, this means that $T$ is a bounded map from 
$L^{p_1} \times L^{p_2}$ to $L^p$ with  $1<p_1,p_2<\infty$ and 
\begin{equation}
\frac{1}{p_1} + \frac{1}{p_2} = \frac{1}{p},
\end{equation}
there exists a kernel $K(x,y,z)$ defined away from the diagonal $x=y=z$ 
such that 
\begin{equation}\label{size}
|K(x,y,z)|\lesssim \frac{1}{(|x-y|+|x-z|)^{2n}},
\end{equation}
\begin{equation}\label{regularity}
|\nabla K(x,y,z)|\lesssim  \frac{1}{(|x-y|+|x-z|)^{2n+1}},
\end{equation}
and such that for $f,g\in L^\infty_c$ we have
\begin{equation}\label{association}
T(f,g)(x)=\iint_{\R^{2n}}K(x,y,z)f(y)g(z)\,dydz, \,\,\, x\notin \supp f\cap \supp g.
\end{equation}
See \cite{GT} and the references therein for more on this type of operators.

We will consider the commutators of bilinear \cz \, operators with functions in an appropriate subspace of $BMO$.
Recall that $BMO$ consists of all locally integrable functions $b$ with $\|b\|_{BMO}<\infty$, where
$$\|b\|_{BMO}=\sup_{Q}\dashint_Q |b(x)-\dashint_Q b|\, dx,$$
the supremum is taken over all cubes $Q\in \R^n$, and, as usual, $\dashint_Q b= b_Q$ denotes the average of $b$ over $Q$
$$\dashint_Q b=\frac{1}{|Q|}\int_Q b(x)\,dx.$$  
The relevant subspace of $BMO$ of multiplicative symbols of our focus is $CMO$, which is defined to be the closure of $C_c^\infty(\R^n)$ in the $BMO$ norm.

Given a bilinear \cz, operator $T$ and a function $b$ in  $BMO$, we consider the following commutators with $b$:
$$[b,T]_1(f,g)=T(bf,g)-bT(f,g)$$
and
$$[b,T]_2(f,g)=T(f,bg)-bT(f,g).$$
Furthermore, given ${\bf b}=(b_1,b_2)$ in $BMO\times BMO$, we consider the iterated commutator:
$$[{\bf b},T]=[ [b_2,[b_1,T]_1 ]_2=[b_1,[b_2,T]_2 ]_1.$$
In fact, for bilinear Calder\'on-Zygmund operators $T$ and ${\bf b}=(b_1,b_2)$, we can define $[{\bf b},T]_\al$ for any multi-index ${\bf \al}=(\al_1,\al_2)\in \N^2_0$, formally as
$$[{\bf b}, T]_{\bf \al}(f,g)(x)=$$
$$\iint(b_1(y)-b_1(x))^{\al_1}(b_2(z)-b_2(x))^{\al_2}K(x,y,z)f(y)g(z)\,dydz.$$

Recall that a bilinear operator is said to be (jointly)\footnote{The only notion of compactness in the bilinear setting used here is referred to as {\it joint compactness} in the related previous works, to differentiate it from the weaker notion of {\it separate compactness}. The latter being the compactness of the linear operators obtained when one of the entries in the bilinear one is kept fixed.}
compact if the image of the bi-unit ball 
$$\{(f,g):\|f\|_{L^{p_1}}\leq 1,\|g\|_{L^{p_2}}\leq 1\}$$
 under its action is a precompact set  in $L^p$.
When $1<p_1,p_2<\infty$, $p=\frac{p_1p_2}{p_1+p_2}\geq 1$, 
$\alpha_1, \alpha_2= 0$ or $1$, and ${\bf b}$ in $ CMO\times CMO$, we have that
$$[{\bf b},T]_{\al}:L^{p_1}\times L^{p_2}\ra L^p$$
is a compact bilinear operator; see \cite{BT}.  In this note we will consider what happens on weighted Lebesgue spaces.

Given ${\bf p}=(p_1,p_2)\in (1,\infty)\times (1, \infty)$ and a vector weight  ${\bf w}=(w_1,w_2)$, let
$$\nu_{\bf w}=\nu_{\bf w,p}=w_1^{\frac{p}{p_1}}w_2^{\frac{p}{p_2}}.$$
The vector weight ${\bf w}$ belongs to the class ${\bf A}_{\bf p}$ provided
$$[{\bf w}]_{{\bf A}_{\bf p}}=\sup_Q\left(\dashint_Q \nu_{\bf w}\right)\left(\dashint_Q w_1^{1-p_1'}\right)^{\frac{p}{p_1'}}\left(\dashint_Q w_2^{1-p_2'}\right)^{\frac{p}{p_2'}}<\infty.$$
In \cite{LOPTT}, Lerner et al  proved that
\begin{equation}\label{vecAp} {\bf w}\in {\bf A}_{\bf p} \Leftrightarrow \left\{\begin{array}{l} \nu_{\bf w}\in A_{2p} \\ \sigma_1=w_1^{1-p_1'}\in A_{2p_1'} \\ \sigma_2=w_2^{1-p_2'}\in A_{2p_2'}. \end{array} \right.\end{equation}
Recall that the classical Muckenhoupt class $A_p$ consists of non-negative weights $w$,  which are locally integrable and such that
\begin{equation}\label{Ap}
[w]_{A_p}=\sup_Q \left(\dashint_Q w\right)\left(\dashint_Q w^{1-p'}\right)^{\frac{p}{p'}}<\infty.
\end{equation}
The weights in the class ${\bf A}_{\bf p}$ characterize the boundedness of the maximal function
$$
\M: L^{p_1}(w_1)\times L^{p_1}(w_1) \to L^{p}(\nu_{\bf w}),
$$
where 
$$
\M(f, g)(x)=\displaystyle\sup_{Q\ni x} \Big(\dashint_Q|f(y)|\,dy\Big)\Big(\dashint_Q|g(z)|\,dz\Big).
$$

From \eqref{vecAp} we can see that when $p\geq 1$ we have
\begin{equation}\label{contained}
A_p\times A_p\subsetneq A_{\min(p_1,p_2)}\times A_{\min(p_1,p_2)}\subsetneq A_{p_1}\times A_{p_2}\subsetneq {\bf A}_{\bf p}.
\end{equation}
The first two containments follow from well known properties of the (scalar) $A_p$ classes and the last  containment is proved in \cite{LOPTT} (see Section \ref{closing} for a new example of the strictness of this containment).
Moreover, we also note that  
\begin{equation}\label{givesAp}
{\bf w}\in A_p\times A_p \implies \nu_{\bf w} \in A_p.
\end{equation}
Indeed,
by H\"older's inequality
$$\Big(\dashint_Q \nu_{\bf w,p}\Big)\Big(\dashint_Q\nu_{\bf w,p}^{1-p'}\Big)^{p-1}\leq [w_1]_{A_p}^{\frac{p}{p_1}}[w_2]_{A_p}^{\frac{p}{p_2}}.$$

It  was shown in \cite{LOPTT} that if $w\in {\bf A}_{\bf p}$ and $T$ is a bilinear Calder\'on-Zygmund operator, then $T$ is bounded from $L^{p_1}(w_1)\times L^{p_2}(w_2)$ into $L^p(\nu_{\bf w})$ and the same result holds for the first order commutator. The boundedness of the iterated commutator on 
weighted Lebesgue spaces in the case of ${\bf A}_{\bf p}$ weights was obtained by P\'erez et al in \cite{PPTT}. The case of product of classical weights 
was considered also by Tang \cite{tang}.

The goal of this paper is to show that the improving effect of the bilinear commutators caries over to the weighted setting when we consider the ``appropriate'' class of weights.  We have the following theorem.

\begin{theorem} \label{singlecomm} Suppose ${\bf p}\in (1,\infty)\times (1, \infty)$, $p=\frac{p_1p_2}{p_1+p_2}> 1$, $b\in CMO$, and ${\bf w}\in A_p\times A_p$. Then $[b,T]_1$ and $[b,T]_2$ are compact operators from $L^{p_1}(w_1)\times L^{p_2}(w_2)$ to $L^p(\nu_{\bf w})$.
\end{theorem}

A similar result holds also for the iterated commutator.

\begin{theorem} \label{itercomm} 
Suppose ${\bf p}\in (1,\infty)\times (1, \infty)$, $p=\frac{p_1p_2}{p_1+p_2}> 1$, ${\bf b}\in CMO\times CMO$, and 
${\bf w}\in A_{p}\times A_p$. Then $[{\bf b},T]$ is a compact operator from $L^{p_1}(w_1)\times L^{p_2}(w_2)$ to $L^p(\nu_{\bf w})$.
\end{theorem}

The remainder of the paper is structured as follows. In Section \ref{proof} we give the proofs of Theorems \ref{singlecomm} and \ref{itercomm}, while in Section \ref{closing} we provide a discussion regarding the class of weights assumed in our main results.

\subsection{Acknowledgement} The authors are grateful for several productive conversations with David Cruz-Uribe and Carlos P\'erez that improved the quality of this article.

\section{Proofs of the theorems}\label{proof}
As pointed out in \cite{CC},  in the linear setting the idea of considering truncated operators to prove compactness results goes back to Krantz and Li \cite{KL}.  We will follow this approach too, but we find convenient to introduce a smooth truncation. (This approach could also be used to simplify some of the computations in \cite{CC} in the linear case.)

 Let $\varphi=\varphi(x,y,z)$ be a non-negative function in $C_c^\infty(\R^{3n})$,  
 $$ \supp \varphi \subset \{(x,y,z):\max(|x|,|y|,|z|)<1\}$$ 
 and such that $$\int_{\R^{3n}} \varphi(u)\, du=1.$$
For $\delta>0$ let $\chi^\delta=\chi^\delta(x,y,z)$ be the characteristic function of the set $$\{(x,y,z):\max(|x-y|,|x-z|)\geq \frac{3\delta}{2} \},$$
and let 
$$\psi^\delta=\varphi_\delta * \chi^\delta,$$ where $$\varphi_\delta(x,y,z)= ( \delta/4)^{-3n} \varphi_\delta(4x/\delta,4y/\delta,4z/\delta).$$
Clearly we have that 
$\psi^\delta \in C^\infty$, 
$$\supp \psi^\delta  \subset  \{(x,y,z):  \max(|x-y|,|x-z|) \geq \delta \},$$  
$\psi^\delta (x,y,z)= 1$ if $\max(|x-y|,|x-z|)> 2 \delta$, and $\|\psi^\delta\|_{L^\infty}\leq 1$.  
Moreover, $\nabla \psi^\delta$ is not zero only if $\max(|x-y|,|x-z|)\approx \delta$
and $\|\nabla \psi^\delta\|_{L^\infty}\lesssim 1/\delta$.  Given a kernel $K$ associated to a \cz \, operator $T$, we define the truncated kernel
$$
K^\delta(x,y,z)= \psi^\delta (x,y,z) K(x,y,z).
$$
It follows that $K^\delta$ satisfies the same size and regularity  estimates of $K$, \eqref{size} and  \eqref{regularity},  with a constant $C$ independent of $\delta$.
%
%
We let  $T^\delta(f,g)$ be the operator defined pointwise by $K^\delta$ through \eqref{association}, now for all $x\in \R^n$.  
We have the following lemma.

\begin{lemma} \label{approx} 
If ${\bf b}\in C_c^\infty\times C_c^\infty,$ then
$$|[{\bf b},T^\delta]_\al(f,g)(x)-[{\bf b},T]_\al(f,g)(x)|\lesssim \|\nabla b_1\|^{\al_1}_\infty\|\nabla b_2\|_\infty^{\al_2} \delta^{|\al|}\M(f,g)(x).$$
Consequently, if ${\bf w}\in {\bf A}_{\bf p}$ we have
$$\lim_{\delta \ra 0}\|[{\bf b},T^\delta]_\al-[{\bf b},T]_\al\|_{L^{p_1}(w_1)\times L^{p_2}(w_2) \ra L^p(\nu_{\bf w})}=0.$$
\end{lemma}

 \begin{proof} 
We adapt the proof  given in  \cite[Lemma 7]{CC} for the linear version of the result. For simplicity we consider the case 
${\bf \alpha}=(1,0)$; the other cases are similar. 
We have,
\begin{align*}
& \left| [ b,T^\delta]_1(f,g)(x)-[{b},T]_1(f,g)(x)\right| \\ 
& \lesssim 
\iint_{ \max(|x-y|,|x-z|) \leq 2\delta}|b(y)-b(x)| |K(x,y,z) f(y)g(z)|\,dydz\\
&
\,\,\, + \iint_{\delta \leq \max(|x-y|,|x-z|) \leq 2\delta}|b(y)-b(x)| |K^\delta(x,y,z)f(y)g(z)|\,dydz\\
&
\lesssim \|\nabla b\|_{L^\infty}
 \iint_{\max(|x-y|,|x-z|) \leq 2\delta} \frac{|f(y)|\,|g(z)|}{  (|x-y| +|x-z|)^{2n-1} }  \,dydz\\
&
\lesssim \|\nabla b\|_{L^\infty}
\sum_{j\geq 0}\iint_{2^{-j}\delta \leq \max(|x-y|,|x-z|) \leq2^{-j+1}\delta}  \frac{|f(y)|\,|g(z)|}{  (|x-y| +|x-z|)^{2n-1} } \,dydz\\
&
\lesssim \|\nabla b\|_{L^\infty}
\sum_{j\geq 0} \left(
\int_{|x-y| \leq2 ^{-j+1}\delta} |f(y)|\,dy 
 \int_{2^{-j}\delta \leq|x-z| \leq2 ^{-j+1}\delta}\frac{ |g(z)|}{  |x-z|^{2n-1} } \,dz
\right. \\
&
\,\,\, + \left.  \int_{2^{-j}\delta \leq|x-y| \leq 2 ^{-j+1}\delta} \frac{|f(y)|}{ |x-y|^{2n-1} }\,dy \int_{ |x-z| \leq2 ^{-j}\delta} |g(z)|\,dz\right) \\
&
\lesssim   \|\nabla b\|_{L^\infty}
\sum_{j\geq 0}  (2^{-j}\delta)^{1-2n}\left( 
 \int_{|x-y| \lesssim 2^{-j+1} \delta} {|f(y)|}\,dy\!\!
\int_{|z-y| \lesssim 2^{-j+1}\delta} {|g(z)|}\, dz\right)\\
&
\lesssim   \|\nabla b\|_{L^\infty} \, \delta \,
\sum_{j \geq 0} 2^{-j} \M(f,g)(x), 
\end{align*}
and the rest of the result follows from the boundedness properties of the maximal function $\M$.
\end{proof}

 Lemma \ref{approx} shows that $[b,T^\delta]_\al$ converges in operator norm to $[b,T^\delta]_\al$ provided the functions $b_1$ and $b_2$ are smooth enough. Therefore, in order to prove that any of the commutators $[{\bf b},T]_\al$ are compact it 
 suffices\footnote{As in the linear case, the limit in the operator norm of compact operators is compact.}
  to work with $[{\bf b},T^\delta]_\al$ for a fixed $\delta$ and our estimates may depend on $\delta$. Moreover, since the bounds of the commutators with $BMO$ functions are of the form 
$$
 \|[{\bf b}, T]_\alpha(f,g)\|_{L^p(\nu_{\bf w})} \lesssim \|b_1\|_{BMO}^{\alpha_1}  |b_2\|_{BMO}^{\alpha_2}   \|f\|_{L^{p_1}(w_1)} \|g\|_{L^{p_2}(w_2) },
$$ 
 to show compactness when working with symbols in $CMO$  we may also assume
  ${\bf b} \in C_c^\infty \times C_c^\infty $ and the estimates may depend on ${\bf b}$ too.

A relevant observation made in \cite[Theorem 5]{CC} is that there exists a sufficient condition for precompactness in $L^r(w)$ when the weight is assumed, crucially for the argument to work, in $A_r$. By adapting the arguments in \cite{HOH}, and, in particular, circumventing the non-translation invariance of $L^r(w)$, the authors in \cite{CC} obtained the following weighted variant of the Frech\'et-Kolmogorov-Riesz result:

\smallskip

\emph{Let $1<r<\infty$ and $w\in A_r$ and let  $\mathcal K\subset L^r(w)$. If 
\begin{enumerate}[{\rm (i)}]
\item $\mathcal K$ is bounded in $L^r(w)$;
\item $\displaystyle\lim_{A\ra \infty} \int_{|x|>A}|f(x)|^r\,w(x) \,dx=0$ uniformly for $f\in \mathcal K$;
\item $\displaystyle\lim_{t\ra 0}\|f(\cdot+t)-f\|_{L^r(w)}=0$ uniformly for $f\in \mathcal K$;
\end{enumerate}
then $\mathcal K$ is precompact in $L^r(w)$.
}

\smallskip

Let us immediately note now that our choice for the class of vector weights in Theorems \ref{singlecomm} and \ref{itercomm} is dictated by the previous compactness criterion. In both our results we will need the weight $\nu_{\bf w, p}\in A_p$ to apply the above version of the Frech\'et-Kolmogorov-Riesz theorem.
In general, if ${\bf w}\in A_{p_1}\times A_{p_2}$ or ${\bf w}\in {\bf A}_{\bf p}$, the best class that $\nu_{\bf w, p}$ belongs to is $A_{2p}$. However, 
as we noticed in \eqref{givesAp}, if ${\bf w}\in A_{p}\times A_{p}$ then $\nu_{\bf w, p}$ is actually in $A_{p}$.  We also point out there there exists examples with ${\bf w\in  A_p}$ and $\nu_{\bf w}\in A_p$, but ${\bf w}\notin A_p\times A_p$ (see Section \ref{closing}).


\begin{proof}[Proof of Theorem \ref{singlecomm}] We will work with the commutator in the first variable; by symmetry, the proof for the other commutator is identical. 
As already pointed out, we may fix $\delta >0$, assume $b\in C_c^\infty$, and study $[b,T^\delta]_1$.  
Suppose $f,g$ belong to
$$B_1(L^{p_1}(w_1))\times B_1(L^{p_2}(w_2))=\{(f,g): \|f\|_{L^{p_1}(w_1)}, \|g\|_{L^{p_2}(w_2)}\leq 1\},$$ with $w_1$ and $w_2$ in $A_p$. 
We need to show that the following conditions hold:

\smallskip

\begin{enumerate}[{\rm (a)}]
\item $[b, T^\delta]_1(B_1(L^{p_1}(w_1))\times B_1(L^{p_2}(w_2))$ is bounded in $L^p(\nu_{\bf w})$;
\item $\displaystyle\lim_{R\ra \infty}\int_{|x|>R}[b, T^\delta]_1 (f,g)(x)^p\nu_{\bf w}\,dx=0$;
\item $\displaystyle\lim_{t\ra 0}\|[b, T^\delta]_1 (f,g)(\cdot+t)-[b, T^\delta]_1 (f,g)\|_{L^p(\nu_{\bf w})}=0.$
\end{enumerate}

\smallskip

It is clear that the first condition (a) holds since 
$$[b,T^\delta]_1: L^{p_1}(w_1)\times L^{p_2}(w_2)\to L^p(\nu_{\bf w})$$
 is bounded when ${\bf w}\in A_p\times A_p\subset {\bf A}_{\bf p}$.

We now show that the second condition (b) holds. It is worth pointing out that for our calculations to work, we need the restrictive assumption $\nu_{\bf w} \in A_p$ which holds since ${\bf w}\in A_p\times A_p$. Let $A$ be large enough so that $\supp b\subset B_A(0)$ and let $R\geq \max (2A,1)$. Then, for $|x|>R$ we have
\begin{align*}
|[b,T^\delta]_1(f,g)(x)|&\leq \|b\|_\infty\int_{\supp b} \int_{\R^{n}}\frac{|f(y)||g(z)|}{(|x-y|+|x-z|)^{2n}}\,dydz \\
&\lesssim \|b\|_\infty\int_{\supp b} |f(y)|\int_{\R^{n}}\frac{|g(z)|}{(|x|+|x-z|)^{2n}}\,dydz\\
&\leq \|b\|_\infty  \|f\|_{L^{p_1}(w_1)}\sigma_1(B_A(0))^{1/p_1'}\!\!\int_{\R^{n}}\frac{|g(z)|}{(|x|+|x-z|)^{2n}}\,dz\\
&\leq \frac{\|b\|_\infty}{|x|^n}\|f\|_{L^{p_1}(w_1)}\sigma_1(B_A(0))^{1/p_1'}\int \frac{|g(z)|}{(|x|+|x-z|)^{n}}\,dz.
\end{align*}

Now, since $|x|>1$, it follows that $|z|+1 \lesssim |z-x|+|x|$
and
$$\int_{\R^{n}}\frac{|g(z)|}{(|x|+|x-z|)^{n}}\,dz \lesssim \|g\|_{L^{p_2}(w_2)}\left(\int_{\R^n}\frac{\sigma_2(z)}{(1+|z|)^{np_2'}}\,dz\right)^{1/p_2'}.$$
Since $w_2 \in A_p \subset  A_{p_2}$, we have $\sigma_2  \in A_{p_2'}$, and hence 
$$\int_{\R^n}\frac{\sigma_2(z)}{(1+|z|)^{np_2'}}\,dz < \infty ;$$
see for example \cite[p. 412]{GCRdF} or \cite[p. 209]{stein}. 
It follows that for $|x|>R$, 
$$
|[b,T^\delta]_1(f,g)(x)| \lesssim \frac{1}{|x|^n}.
$$

Raising both sides of the last inequality to the power $p$ and integrating over $|x|>R$ we have
$$\int_{|x|>R}|[b,T^\delta]_1(f,g)(x)|^p\nu_{\bf w}\,dx\lesssim_{b,{\bf p},{\bf w}} \int_{|x|>R}\frac{\nu_{\bf w}(x)}{|x|^{np}}\,dx\ra 0, \quad R\ra \infty,$$
where we used again the fact that for $v \in A_r$, $r>1$,
$$
     \int_{\R^n}\frac{v(x)}{(1+|x|)^{nr}}\,dx < \infty.
$$

We now show the uniform equicontinuity estimate (c).  Note that
\begin{align*}\lefteqn{ [b,T^\delta]_1(f,g)(x+t)-[b,T^\delta]_1(f,g)(x)}\\
&=\iint_{\R^{2n}}(b(y)-b(x+t))K^\delta(x+t,y,z)f(y)g(z)\,dydz\\
&\qquad \qquad \qquad \qquad -\iint_{\R^{2n}}(b(y)-b(x))K^\delta(x,y,z)f(y)g(z)\,dydz\\
&= (b(x)-b(x+t))\iint_{\R^{2n}}K^\delta(x,y,z)f(y)g(z)\,dydz  \,\,\,\,+\\
& \,\, \iint_{\R^{2n}}(b(y)-b(x+t))(K^\delta(x+t,y,z)-K^\delta(x,y,z))f(y)g(z)\,dydz\\
&=I_1(t,x)+I_2(t,x).
\end{align*}
To estimate $I_1$, we observe first that
$$|I_1(t,x)|\leq |t|\|\nabla b\|_\infty \tilde T_*(f,g)(x),$$
where  $\tilde T_{*}(f, g)$ denotes the
maximal truncated bilinear singular integral operator
$$
\tilde T_{*}(f, g)(x) = \sup_{\delta>0}|T^\delta(f,g)(x)| = \sup_{\delta>0}\left|\iint_{\R^{2n}}K^\delta (x, y, z)f(y)g(z)\,dydz\right|.
$$
Note that with similar arguments to the ones used in the proof of Lemma~\ref{approx},
$$
\left| T^\delta(f,g)(x) - \iint_{\max (|x-y|,|x-z|)\geq \delta}K (x, y, z)f(y)g(z)\,dydz\right|
\lesssim
$$
$$
\left|  \iint_{\delta<\max (|x-y|,|x-z|)\leq 2\delta}\frac{|f(y)g(z)|}{(|x-y|+|x-z|)^{2n}}\,dydz\right| \lesssim \M(f,g)(x).
$$

It follows then that 
\begin{equation}\label{equivalentT}
\tilde T_{*}(f, g)(x) \lesssim T_{*}(f, g)(x) + \M(f,g)(x),
\end{equation}
where now
$$
 T_{*}(f, g)(x) = \sup_{\delta>0}\left|\iint_{\max (|x-y|,|x-z|)\geq \delta } K (x, y, z)f(y)g(z)\,dydz\right|.
$$

By the pointwise estimate \cite[(2.1)]{GT2}, for all $\eta>0$
\begin{equation}\label{old}
 T_{*}(f, g)(x) \lesssim_\eta (M(|T(f,g)|^\eta)(x))^{1/\eta} +  Mf(x)Mg(x),
\end{equation}
where $M$ is the Hardy-Littlewood maximal function. From \eqref{equivalentT} and \eqref{old} (with $\eta =1$ in our current situation)
 it easily follows that
$$
\tilde T_{*} : L^{p_1}(w_1) \times  L^{p_2}(w_2) \ra L^p(\nu_{\bf w})
$$
for ${\bf w}\in A_p \times A_p$. We obtain then
$$\|I_1(t,x)\|_{L^p(\nu_{\bf w})}\lesssim |t|.$$

To estimate $I_2$, we observe that, if $t<\delta/4$,  
$$K^\delta(x+t,y,z)-K^\delta(x,y,z)=~0$$
  when  $\max(|x-y|, |x-z|) \leq \delta/2$. Therefore,
with what are by now familiar arguments,
 \begin{align*}
& \left| \iint (b(y)-b(x+t))(K^\delta(x+t,y,z)-K^\delta(x,y,z))f(y)g(z)\,dydz \right| \\
&\lesssim \|b\|_\infty|t|\iint_{\max\{|x-y|,|x-z|\}>\delta/2}\frac{|f(y)||g(z)|}{(|x-y|+|x-z|)^{2n+1}}\,dydz\\
&\lesssim \|b\|_\infty|t| \sum_{j\geq 0}\iint_{2^{j-1}\delta < \max\{|x-y|,|x-z|\} \leq 2^{j}\delta}\frac{|f(y)||g(z)|}{(|x-y|+|x-z|)^{2n+1}}\,dydz\\
&\lesssim \frac{|b\|_\infty|t|}{\delta}\M(f,g)(x).
\end{align*}
From the boundedness properties of $\M$ we obtain the desired estimate
$$\|I_2(t,x)\|_{L^p(\nu_{\bf w})}\lesssim |t|.$$
\end{proof}

We concentrate now on the compactness of the iterated commutator. We will show that $[{\bf b}, T^\delta]$ satisfies the corresponding conditions (a), (b) and (c) listed at the beginning of the proof of Theorem \ref{singlecomm}.  The proof is similar to that of Theorem \ref{singlecomm}, but  it is worth pointing out that for the iterated commutator, these conditions hold under the weakest assumption on the class of weights, that is, ${\bf w}\in {\bf A}_{{\bf p}}$. We indicate the needed modifications in the proof below.

\begin{proof}[Proof of Theorem \ref{itercomm}]

As before, we may assume ${\bf b} \in C_c^{\infty}\times C_c^\infty$, fix $\delta>0$ and study $[{\bf b}, T^\delta]$.
Once again, condition (a) holds since $[{\bf b}, T^\delta]$ is bounded from $L^{p_1}(w_1)\times L^{p_2}(w_2)$ to $L^p(\nu_{{\bf w}})$ when ${\bf w} \in {\bf A}_{\bf p}$. Next, we show that condition (b) holds. Let $A$ be large enough so that $\supp b_1 \cup \supp b_2 \subset B_A(0)$ and let $R\geq 2A$. Then, for $|x|\geq R$, we have

\begin{equation*}\begin{split}
&|[{\bf b}, T^{\delta}](f,g)(x)| \lesssim \|b_1\|_{\infty} \|b_2\|_{\infty} \int_{\supp b_1} \int_{\supp b_2} \frac{|f(y)| |g(z)|}{(|x-y|+|x-z|)^{2n}} dy dz \\
&\lesssim \frac{1}{|x|^{2n}} \|b_1\|_{\infty} \|b_2\|_{\infty} \int_{\supp b_1} |f(y)| dy \int_{\supp b_2} |g(z)| dz \\
&\lesssim \frac{1}{|x|^{2n}} \|b_1\|_{\infty} \|b_2\|_{\infty} \|f\|_{L^{p_1}(w_1)} \|g\|_{L^{p_2}(w_2)} \sigma_1(\supp b_1)^{1/{p_1}'} \sigma_2(\supp b_2)^{1/p_2'}.
\end{split}\end{equation*}

We can raise the previous pointwise estimate to the power $p$ and integrate over $|x|>R$ to get
$$\int_{|x|>R} |[{\bf b}, T^{\delta}](f,g)(x)|^p\nu_{\bf w} (x)\,dx$$
$$\leq \left(\|f\|_{L^{p_1}(w_1)} \|g\|_{L^{p_2}(w_2)} \sigma_1(\supp b_1)^{1/p_1'} \sigma_2(\supp b_2)^{1/p_2'}\right)^p
\!\!\int_{|x|>R} \frac{\nu_{{\bf w}}(x)}{|x|^{2np}} dx,
$$
which tends to zero as $R \to \infty$  even if $\nu_{{\bf w}} \in A_{2p}$, and gives (b).

To show that condition (c) also holds, we write
\begin{equation*}\begin{split}
&|[{\bf b}, T^{\delta}](f,g)(x) - [{\bf b}, T^{\delta}](f,g)(x+t)| = \\
&\left| \iint_{\R^{2n}} (b_1(y)-b_1(x)) (b_2(z)-b_2(x)) K^{\delta}(x,y,z) f(y) g(z) dy dz \,\,\, +\right. \\
& \left. \iint_{\R^{2n}} \!\!(b_1(y)-b_1(x+t)) (b_2(z)-b_2(x+t)) K^{\delta}(x+t,y,z) f(y) g(z) dy dz \right| \\
&\leq |I_1(x,t)| + |I_2(x,t)|,
\end{split}\end{equation*}
where
\begin{equation*}
I_1(x,t)= (b_1(x+t)-b_1(x))\!\! \iint_{\R^{2n}} \!\!(b_2(z)-b_2(x)) K^{\delta}(x,y,z)f(y)g(z) dy dz
\end{equation*}
and
\begin{equation*}\begin{split}
&I_2(x,t)=\\
&\iint_{\R^{2n}} (K^{\delta}(x,y,z)(b_2(z)-b_2(x))- K^{\delta}(x+t,y,z)(b_2(z)-b_2(x+t))) \\
&\qquad\qquad\qquad\qquad\qquad\qquad\qquad\times (b_1(y)-b_1(x+t)) f(y) g(z) dy dz.\\
&
\end{split}\end{equation*}

The pointwise estimate of $I_1(x,t)$ can be obtained as in the proof of Theorem \ref{singlecomm},
$$|I_1(x,t)|\leq |t| \|\nabla b_1\|_{\infty}( \tilde T^*(f,b_2g)(x)+\|b_2\|_\infty \tilde T^*(f,g)(x)).$$
To invoke now the boundedness 
$$
\tilde T_{*} : L^{p_1}(w_1) \times  L^{p_2}(w_2) \ra L^p(\nu_{\bf w})
$$
for all ${\bf w} \in {\bf A}_{\bf p}$ and not just ${\bf w} \in {A}_{p} \times {A}_{p}$, we can use instead of 
\eqref{old} a strengthened version of it. Namely,  
\begin{equation}\label{new}
 T_{*}(f, g)(x) \lesssim_\eta (M(|T(f,g)|^\eta)(x))^{1/\eta} +  \M(f,g)(x),
\end{equation}
which is implicit in the arguments in \cite{GT2} and explicit in the article by Chen \cite[(2.1)]{chen}.
Thus, as $|t|\ra 0$,
\begin{equation*}
\|I_1\|_{L^p(\nu_{{\bf w}})} \lesssim  |t|\|\nabla b_1\|_{\infty} \| b_2\|_{\infty} \|f\|_{L^{p_1}(w_1)} \|g\|_{L^{p_2}(w_2)} \longrightarrow 0.
\end{equation*}
Now, we split $I_2$ in two other terms as follows
\begin{equation*}\begin{split}
&I_2(x,t) = \iint_{\R^{2n}} (K^{\delta}(x,y,z)-K^{\delta}(x+t,y,z)) (b_2(z)-b_2(x+t)) \times \\
&\qquad\qquad\qquad\qquad\qquad  \,\,\,\,\,  \qquad\qquad \times(b_1(y)-b_1(x+t))f(y)g(z)dydz \\
&+ \,\,\, (b_2(x+t)-b_2(x)) \iint_{\R^{2n}} (b_1(y)-b_1(x+t)) K^{\delta}(x,y,z) f(y)g(z) dydz \\
&= I_{21}(x,t) + I_{22}(x,t).
\end{split}\end{equation*}
As in Theorem \ref{singlecomm}, the estimate of $I_{21}$, for $t$ sufficiently small reduces to
\begin{align*}
|I_{21}&(x,t)|  \lesssim  \\
& \lesssim |t| \|b_1\|_{\infty} \|b_2\|_{\infty} \iint_{  \max\{|x-y|,|x-z|\}>\delta/2   }\frac{|f(y)| |g(z)|}{(|x-y|+|x-z|)^{2n+1}} \, dydz \\
& \lesssim \frac{  |t| }{\delta} \| b_1\|_{\infty} \|b_2\|_{\infty} \M(f, g)(x),
\end{align*}
which is again an appropriate estimate to obtain (c).
Finally, 
$$|I_{22}(x,t)|  \leq  $$$$
|t| \|\nabla b_2\|_{\infty} \left| \iint_{\max(|x-y|,|x-z|)\geq \delta} \!\!\!\!(b_1(y)-b_1(x+t)) K^\delta(x,y,z) f(y) g(z) dy dz \right|$$
$$\leq |t| \|\nabla b_2\|_{\infty} (\tilde T^*(b_1f,g)(x) + \|b_1\|_{\infty} \tilde T^*(f,g)(x)).$$
Therefore, as $|t|\ra 0$,
\begin{equation*}
\|I_{22}\|_{L^p (\nu_{{\bf w}})} \lesssim |t| \|\nabla b_2\|_{\infty} \|b_1\|_{\infty} \|f\|_{L^{p_1}(w_1)} \|g\|_{L^{p_2}(w_2)} \longrightarrow 0.
\end{equation*}
\end{proof}

\section{Closing remarks}\label{closing}

1. Our results on bilinear commutators highlight one more time the fact that the higher the order of the commutator with $CMO$ symbols, the less singular the operators are. In this article this is reflected in the less restrictive class of weights needed to achieve the estimates (a), (b) and (c). Indeed, in Theorem~\ref{singlecomm}, the assumption $A_p\times A_p$ on the weight is needed both to check condition (b) and to guarantee that the target weight falls in the right class. However, to obtain bilinear compactness in Theorem~ \ref{itercomm} we require the $A_p\times A_p$  assumption about the vector weight only because the sufficient condition from \cite{CC} on $L^p(\nu_{\bf w})$ precompactness requires $\nu_{\bf w}\in A_p$. As already mentioned,  this last condition fails if ${\bf w}$ is only assumed to belong to ${\bf A}_{\bf p}$. Actually, our techniques can be used to obtain a more general theorem by assuming that ${\bf w}\in {\bf A_p}$ and $\nu_{\bf w}\in A_p$ instead of ${\bf w}\in A_p\times A_p$.   
\begin{theorem} \label{singlecomm2} Suppose ${\bf p}\in (1,\infty)\times (1, \infty)$, $p=\frac{p_1p_2}{p_1+p_2}> 1$, $b\in CMO$, and
${\bf w}\in {\bf A_p}$ with $\nu_{\bf w}\in A_p$. Then $[b,T]_1$ and $[b,T]_2$ are compact operators from $L^{p_1}(w_1)\times L^{p_2}(w_2)$ to $L^p(\nu_{\bf w})$.
\end{theorem}

\begin{theorem} \label{itercomm2} 
Suppose ${\bf p}\in (1,\infty)\times (1, \infty)$, $p=\frac{p_1p_2}{p_1+p_2}> 1$, ${\bf b}\in CMO\times CMO$, and
${\bf w}\in {\bf A_p}$ with $\nu_{\bf w}\in A_p$. Then $[{\bf b},T]$ is a compact operator from $L^{p_1}(w_1)\times L^{p_2}(w_2)$ to $L^p(\nu_{\bf w})$.
\end{theorem}

 As mentioned in the introduction 
 $${\bf w}\in A_p\times A_p \Rightarrow {\bf w\in A_p} \ \text{and} \ \nu_{\bf w}\in A_p.$$ 
 To see that the assumption ${\bf w\in A_p}$ and $\nu_{\bf w}\in A_p$ is indeed weaker, consider the example $w_1(x)=|x|^{-\al}$ where $1<\al<\frac{p_1}{p}=1+\frac{p_1}{p_2}$ and $w_2(x)=1$ on $\R$.  Then $\sigma_1(x)=|x|^{\al(p_1'-1)}$ belongs to $A_{2p_1'}$ since 
 $$\al<1+\frac{p_1}{p_2}<1+p_1=\frac{2p_1'-1}{p_1'-1}.$$
Moreover, $\nu_{\bf w}(x)=|x|^{-\al\frac{p}{p_1}}$ belongs to $A_1 (\subset A_p)$ since $\al\frac{p}{p_1}<1$.  However, the weight $w_1$ does not belong to any $A_p$ class since it is not locally\, integrable.  This vector weight also provides a new example of the properness of the containment $A_{p_1}\times A_{p_2}\subsetneq {\bf A_p}$ from \cite[Section 7]{LOPTT}.
\smallskip

2. 
It is natural to ask whether the sufficient condition about $L^p(w)$ precompactness in \cite{CC} may be extended to include weights $w\in A_q$ with $q>p$. A simple modification of the argument in \cite[p. 275]{Y} gives the following result in this setting:

\smallskip

\emph{Let $1<r<\infty$, $w\in A_\infty$, and $\mathcal K\subset L^r(w)$. If
\begin{enumerate}[{\rm (I)}]
\item $\mathcal K$ is bounded in $L^r(w)$;
\item $\displaystyle\lim_{A\ra \infty} \int_{|x|>A}|f(x)|^r\,wdx=0$ uniformly for $f\in \mathcal K$;
\item $\displaystyle\|f(\cdot+t_1)-f(\cdot+t_2)\|_{L^r(w)}\to 0$ uniformly for $f\in \mathcal K$ as $|t_1-t_2|\to 0$;
\end{enumerate}
then $\mathcal K$ is precompact.
}

\smallskip

This is different than the sufficient condition we employed in the proofs of our main theorems, specifically in the third assumption about equicontinuity. Note that, in general, the non-translation invariance of the measure deems our last condition strictly stronger than the corresponding one in \cite{CC}. Unfortunately, the arguments we used to prove Theorem \ref{itercomm} do not seem to hold anymore in this setting.

\end{document}